\documentclass[12pt,reqno]{amsart}
\usepackage{graphicx}
\newtheorem{thm}{Theorem}[section]

\theoremstyle{definition}

\theoremstyle{remark}

\numberwithin{equation}{section}

\begin{document}

\title[Burnside's formula]{Sharp inequalities related with Burnside's formula}%
\author{necdet batir}%
\address{department of  mathematics, nev{\c{s}}ehir hbv university, nev{\c{s}}ehir, 50300 turkey}%
\email{nbatir@hotmail.com}%


\begin{abstract} We prove the following double inequality related with Burnside's formula for $n!$
\begin{equation*}
\sqrt{2\pi}\left(\frac{n+a_*}{e}\right)^{n+a_*}<n!<\sqrt{2\pi}\left(\frac{n+a^*}{e}\right)^{n+a^*}\,(n\in\mathbb{N}),
\end{equation*}
where the constants $a_*=0.428844044...$ and $a^*=0.5$ are the best possible. We believe that the method we used in the proof gives insight to undergraduate students to understand  how simple inequalities can be established.
\end{abstract}
\maketitle
\section{introduction}
In a paper written in 1733 the French mathematician Abraham de Moivre developped the formula
\begin{equation*}
n!\sim C\cdot\sqrt{n}n^ne^{-n},
\end{equation*}
where $C$ is a constant. He was unable, however, to evaluate numerical value of this constant; this task befell a Scot mathematician James Stirling (1692-1770), who found  $C=\sqrt{2\pi}$. Formula
\begin{equation}\label{e:1}
n!\sim n^ne^{-n}\sqrt{2\pi n}
\end{equation}
is  known as Stirling's formula today. It is known that this formula has many applications in statistical physics, probability theory and number theory. The most well known approximation formula for factorial function after Stirling formula is Burnside's formula \cite{2}, which is given by
\begin{equation}\label{e:2}
n!\sim\sqrt{2\pi}\left(\frac{n+\frac{1}{2}}{e}\right)^{n+\frac{1}{2}}.
\end{equation}
It is known that Burnside's formula is more accurate than Stirling's formula. The gamma function $\Gamma$ is defined by the improper integral, for $x>0$:
\begin{equation*}
\Gamma(x)=\int\limits_0^\infty t^{x-1}e^{-t}dt.
\end{equation*}
An important function related to $\Gamma$ is the digamma function $\psi$, which is defined by $\psi(x)=\Gamma'(x)/\Gamma(x)$ for $x>0$. The gamma function $\Gamma$ and factorial function $n!$ are related with  $\Gamma(n+1)=n!$, for all $n\in\mathbb{N}$. In the literature there are many inequalities related to Burnside's formula, for example
\begin{equation}\label{e:3}
\sqrt{2e}\left(\frac{n+\frac{1}{2}}{e}\right)^{n+\frac{1}{2}}<n!<e^{(1+\gamma)e^{-\gamma}}\left(\frac{n+e^{-\gamma}}{e}\right)^{n+e^{-\gamma}}.
\end{equation}
see \cite[Theorem 1.4]{1}, and
\begin{equation}\label{e:4}
\sqrt{2e}\left(\frac{n+\frac{1}{2}}{e}\right)^{n+\frac{1}{2}}<n!<\sqrt{2\pi}\left(\frac{n+\frac{1}{2}}{e}\right)^{n+\frac{1}{2}},
\end{equation}
see  \cite[Theorem 1.5]{1}. Lu \cite{4}  provided an asymptotic expansion for the gamma function starting from Burnside's formula. Chen \cite{3} found the following asymptotic expansion starting from Burnside's formula
\begin{equation*}
n!\sim\sqrt{2\pi}\left(\frac{n+\frac{1}{2}-\frac{1}{24}\frac{1}{(n+1/2)^2}+\frac{19}{5760}\frac{1}{(n+1/2)^3}+\cdots}{e}\right)^{n+\frac{1}{2}}.
\end{equation*}
In this short note, motivated by the inequalities (\ref{e:3}) and (\ref{e:4}), we determine the largest number $\alpha$ and the smallest number $\beta$ in such away that the following inequalities hold for all $n\in\mathbb{N}$:
 \begin{equation*}
\sqrt{2\pi}\left(\frac{n+\alpha}{e}\right)^{n+\alpha}<n!<\sqrt{2\pi}\left(\frac{n+\beta}{e}\right)^{n+\beta}.
\end{equation*}
We prove that the best possible constants $\alpha$ and $\beta$ satisfying these inequalities are $\alpha=0.428844044...$ and $\beta=1/2$. The right hand side inequality here is already known, but we want to emphasize that we have shown here that the scaler $a^*=1/2$ here can not be replaced by a smaller quantity. We believe that the methods we used in the proof help undergraduate students to gain insight in establishing simple inequalities.

\section{main result}
Our main result is the following theorem.
\begin{thm} For all $n\in\mathbb{N}$ we have
\begin{equation}\label{e:5}
\sqrt{2\pi}\left(\frac{n+a_*}{e}\right)^{n+a_*}<n!<\sqrt{2\pi}\left(\frac{n+a^*}{e}\right)^{n+a^*},
\end{equation}
where the constants $a_*=0.428844044...$ and $a^*=0.5$ are the best possible.
\end{thm}
\begin{proof} The right hand side of (\ref{e:5}) is already known (see \cite{1}). We assume that the right side of (\ref{e:5}) holds. Then we have for all $n\in\mathbb{N}$:
$$
\sqrt{2\pi}\left(\frac{a_*+n}{e}\right)^{a_*+n}\geq n!
$$
or taking logarithms of both sides, a simple calculation gives
$$
\log \frac{n!}{\sqrt{2\pi n}n^ne^{-n}}+(n+1/2)\log n+a^*-(n+a^*)\log(n+a^*)<0.
$$
Letting $n\to\infty$ and applying Stirling's formula yields
\[
\lim\limits_{n\to\infty}\log\frac{n^{\frac{1}{2}-a^*}e^{a^*}}{\left(1+\frac{a^*}{n}\right)^{n+a^*}}<0 \quad \mbox{or}\quad \lim\limits_{n\to\infty}\left(\frac{1}{2}-a^*\right)\log n<0,
\]
which implies that $a^*\geq \frac{1}{2}$. Now, we assume that  the left side inequality of (\ref{e:5}) is valid for all $n=1,2,,\cdots$. Then if we substitute  $n=1$ there we obtain
\begin{equation*}
\sqrt{2\pi}\left(\frac{a_*+1}{e}\right)^{a_*+1}\leq1
\end{equation*}
or
\begin{equation}\label{e:6}
\frac{1}{2}\log (2\pi)+(a_*+1)\log(a_*+1)-a_*-1\leq0.
\end{equation}
Let us define  for $t\geq0$
$$
g(t)=(t+1)\log(t+1)+\frac{1}{2}\log(2\pi)-t-1.
$$
Then clearly, (\ref{e:6}) is equivalent to
\begin{equation}\label{e:7}
g(a_*)\leq0.
\end{equation}
Since
$$
g(0)=\frac{1}{2}\log (2\pi)-1=-0.0810615...<0
$$
and
$$
g(1)=2\log2+\frac{1}{2}\log (2\pi)-2=0.305233...>0,
$$
and $g$ is strictly increasing on $[0,\infty)$, we conclude that $g$ has only one real root on $(0,1)$, which is $t_0=0.428844\cdots$. We therefore obtain from (\ref{e:7}) $g(a_*)\leq0=g(t_0)$, which implies $a_*\leq t_0$ from fact that $g$ is strictly increasing.

Now we shall show that the left hand side of (\ref{e:5}) holds for $a_*=0.428844044...$. For this reason we define
$$
G_b(t)=\frac{1}{2}\log(2\pi)+(t+b)\log(t+b)-t-b-\log\Gamma(t+1), \quad t\geq 1.
$$
Differentiation gives
$$
G_b'(t)=\log(t+b)-\psi(t+1),
$$
where $\psi$ is the digamma function. By \cite[Lemma 1.7]{1}  $G_{a_*}$ is strictly decreasing. Thus we get for $t\geq 1$
$$
G_{a_*}(t)\leq G_{a_*}(1)=-6.58087\times10^{-11}<0,
$$
which gives the left hand side of (\ref{e:5}).

We conclude that the constants $a_*=0.428844044...$ and $a^*=0.5$ are the best possible constants.
\end{proof}

\bibliographystyle{amsplain}

\end{document}